\documentclass{article}
\usepackage{amsthm,amsmath,amssymb,url}
\usepackage{fullpage}

\newtheorem{theorem}{Theorem}
\newtheorem{proposition}[theorem]{Proposition}
\newtheorem{lemma}[theorem]{Lemma}
\theoremstyle{definition}

\newcommand{\eqdef}{\overset{\mathrm{def}}{=}}

\begin{document}

\title{An Elementary Linear-Algebraic Proof without Computer-Aided Arguments for the Group Law on Elliptic Curves}

\author{Koji Nuida${}^1{}^2$}
\date{${}^1$ Institute of Mathematics for Industry (IMI), Kyushu University, Fukuoka 819-0395, Japan \\
\url{nuida@imi.kyushu-u.ac.jp} \\
${}^2$ National Institute of Advanced Industrial Science and Technology (AIST), Tokyo, Japan}

%\subjclass[2010]{Primary 14H52; Secondary 14G50}

\maketitle

\begin{abstract}
The group structure on the rational points of elliptic curves plays several important roles, in mathematics and recently also in other areas such as cryptography.
However, the famous proofs for the group property (in particular, for its associative law) require somewhat advanced mathematics and therefore are not easily accessible by non-mathematician.
On the other hand, there have been attempts in the literature to give an elementary proof, but those rely on computer-aided calculation for some part in their proofs. 
In this paper, we give a self-contained proof of the associative law for this operation, assuming mathematical knowledge only at the level of basic linear algebra and not requiring computer-aided arguments.
\end{abstract}

\section{Introduction}
\label{sec:intro}

The well-known group structure on rational points of elliptic curves is one of the famous examples of \lq\lq mysterious'' mathematical phenomena that have also attracted many interests from mathematical non-specialists.
One of the reasons is that rational point groups on elliptic curves have practical applications such as so-called Elliptic Curve Cryptography \cite{Kob87,Mil86} and so-called Elliptic Curve Method in integer factorization \cite{Len85}.
As the elliptic curves themselves (when determined concretely by the Weierstrass equations) and the group operation for their rational points are defined in an elementary manner, even people without advanced mathematical knowledge can \emph{use} elliptic curve groups in application.
On the other hand, the existing \emph{proofs} in the literature (to the author's best knowledge) showing that the operation indeed defines a group are not easy for mathematical non-specialists to \emph{understand} by themselves, rather than just \emph{believing} professional mathematicians who proved the group law or computer programs that symbolically verified the group law (see Section \ref{subsec:intro__related_work} below).
Therefore, from not only mathematical but also practical or educational viewpoints, it is worthy to give a proof for the group law over elliptic curves (in particular, the associativity, which is the only part with significant difficulty) that is easier to understand even for mathematical non-specialists.

Towards this goal, in this paper we give a new self-contained proof for the aforementioned associativity, by revisiting a famous proof in the literature and removing the use of advanced mathematical knowledge inside the original proof.
As a result, the required mathematical knowledge in our proof is only at the level of basic linear algebra, and our proof does not require heavy computation that is usually outsourced to computers.

\subsection{Our Result and Related Work}
\label{subsec:intro__related_work}

Before explaining the idea of our proof, we compare the following four famous proof strategies for the group law over elliptic curves.

\paragraph{Using algebraic geometry.}

From the viewpoint of algebraic geometry, the operator defined over rational points of an elliptic curve $E$ satisfies the group law because it is naturally isomorphic to the degree-$0$ part $\mathrm{Pic}^0(E)$ of the Picard group of $E$ (see e.g., \cite[Proposition III.3.4]{Silverman}).
It is an elegant proof, but it relies on many advanced mathematics such as those behind the Riemann--Roch Theorem used in the proof.

\paragraph{Using complex analysis.}

There is also a relation between the Weierstrass $\wp$ function and elliptic curves $E$ over the complex field $\mathbb{C}$, which also naturally induces the group structure on $E$ (see e.g., \cite[Section 2.2]{Silverman+}).
This direction may be accessible for people who are familiar with complex analysis.
However, to \lq\lq transfer'' the result holding only over $\mathbb{C}$ to arbitrary fields, we need some other machinery such as the Lefschetz's Principle (see e.g., \cite{Ekl73}) which requires advanced knowledge of mathematical logic.

\paragraph{Using direct calculation.}

As the operation for rational points is described by concrete rational functions in coordinates of the original points, it is in principle possible to verify the associativity just by direct calculation.
This should be the most elementary proof if succeeded, and there have been some attempts with this direction \cite{Fri17,Rus17,The07}.
However, the papers \cite{Rus17,The07} are focusing on computer-aided formal proofs of the group law, rather than proving it by hands.
To the author's best knowledge, the work by Friedl \cite{Fri17} is the closest in the literature to the complete success in this direction.
However, even in that paper, the detail for the most complicated part (at the end of Lemma 2.1) is omitted by just saying that it is verified by computer, therefore the proof is still not entirely hand-made.
Such a computer-aided proof should not be unreasonably undervalued; but for the aim of the present work, this is in some sense just changing to relying on computer instead of relying on mathematicians who proved, e.g., the Riemann--Roch Theorem.

\paragraph{Using the Cayley--Bacharach Theorem.}

This proof (see e.g., \cite[Appendix A]{Silverman+}) is in fact the starting point of the present work.
Let $P$, $Q$, and $R$ be rational points of an elliptic curve $E$.
Take eight points $P_1,\dots,P_8$, \emph{possibly with multiplicity}, that appear during the computation of $(P + Q) + R$ and $P + (Q + R)$, and let $P_9$ and $P_{10}$ be points associated to $(P + Q) + R$ and to $P + (Q + R)$, respectively.
Then we can take two cubic curves $F_1$ and $F_2$ (each being the union of three lines) in such a way that $F_1$ passes through $P_1,\dots,P_8,P_9$, and $F_2$ passes through $P_1,\dots,P_8,P_{10}$.
Now, as both $F_1$ and $F_2$ contain the eight points $P_1,\dots,P_8$ common to $E$, the Cayley--Bacharach Theorem implies that the sets (\emph{possibly with multiplicity}) $\{P_1,\dots,P_8,P_9\}$ and $\{P_1,\dots,P_8,P_{10}\}$ must coincide, therefore $P_{10} = P_9$ and $P + (Q + R) = (P + Q) + R$.

When $P_1,\dots,P_8,P_9$ are all distinct, the aforementioned form of the Cayley--Bacharach Theorem can be stated and proved by using linear algebra only, as suggested in \cite[Section 1.2]{Silverman+} and indeed done in the present paper.
However, a difficulty arises when some of these points coincide with each other so that some \emph{multiplicity} occurs.
In such a general case, a rigorous statement of the Cayley--Bacharach Theorem (even specialized to the current situation) is described by using the notion of intersection multiplicity of two curves; and the intersection multiplicity is defined by using the notion of local rings, which is somewhat less elementary in comparison to basic linear algebra.

\paragraph{Outline and properties of our proof.}

As mentioned above, our proof here has close connection to the proof based on the Cayley--Bacharach Theorem.
Recall that the difficulty in the original proof arises when dealing with intersection points of two cubic curves with multiplicity larger than one.
In more detail, we observe that the difficulty originates precisely in distinguishing the intersection points with multiplicity three from those with multiplicity two, which is necessary in the original proof for a general case.
Our first idea is that by using somewhat tricky combinatorial arguments (in Section \ref{sec:many_coincidence} below), the case involving points of multiplicity three can be reduced to the case where all the points have multiplicity at most two.
Then our second idea is that the condition of the intersection multiplicity being two (or larger) can be formulated in terms of relations between the curve $F_i$ and tangent lines of $E$, which is described just by using (formal) derivatives and is easily fitted to the linear-algebraic framework as in the aforementioned case of distinct points.
The overall proof is organized by using case-by-case analyses, but the number of cases to be considered is still not very large.

We also note that our proof does not use any particular property of the coefficient field $K$ and hence is applicable to an arbitrary field $K$.
In more detail, first, our proof deals with a general Weierstrass equation $y^2 + a_1 x y + a_3 = x^3 + a_2 x^2 + a_4 x + a_6$ directly, rather than working on its Weierstrass normal form $y^2 = x^3 + a_4 x + a_6$ (as in \cite{Fri17}, for example) which is in general not available when $\mathrm{char}(K) \in \{2,3\}$.
Secondly, our proof does not assume that $K$ is algebraically closed (in contrast to some previous proofs such as \cite{Fri17}), therefore even the notion of algebraic closure of a field is not needed.

We also mention about another related work in Washington's book \cite[Section 2.4]{Washington}.
The proof of the associativity given in that book is also similar to our proof; it becomes the same as the proof in \cite[Appendix A]{Silverman+} when the points $P_i$ appearing in the proof are all distinct.
For the other case where non-trivial multiplicity occurs, in contrast to our proof, the proof in the book deals with the difference between multiplicities two and three in some algebraic way that is more elementary than the theory of local rings used in \cite[Appendix A]{Silverman+}.
As a result, the proof in the book is \emph{almost} relying on linear algebra only.
We note, however, that there are the following two differences compared to our proof.
First, in contrast to our proof applicable to a general Weierstrass form in a unified way, the argument in \cite[Section 2.4]{Washington} assumes an expression of the given elliptic curve in short Weierstrass form only (though it is claimed that the proof is similarly extendible to more general cases).
Secondly, a part of that proof (Lemma 2.7 in the book) requires the coefficient field $K$ to have at least three elements; accordingly, an extension field should be used in the case $K = \mathbb{F}_2$.
In contrast, our argument does not require any such condition for the field $K$ and even the algebraic notion of extension fields is not needed.
From those viewpoints, we can say that our proof is more general and elementary than the proof in \cite[Section 2.4]{Washington}.

\paragraph{Acknowledgements.}

The author thanks Go Yamashita and Tsuyoshi Takagi for their valuable comments.
The author also thanks the anonymous reviewer for the careful review, especially for pointing out the related work in \cite{Washington}.
This work is supported by JST CREST Grant Number JPMJCR14D6 and JSPS KAKENHI Grant Number JP19H01804.

\section{Preliminaries}
\label{sec:preliminaries}

In this section, we summarize some basic properties of elliptic curves and fix notations used in the paper.
Let $K$ be an arbitrary field.
We consider a (smooth) elliptic curve $E$ over $K$ in the projective plane $\mathbb{P}^2$ over $K$ defined by a Weierstrass equation of the form
\[
E \colon X^3 + a_2 X^2 Z + a_4 X Z^2 + a_6 Z^3 - Y^2 Z - a_1 X Y Z - a_3 Y Z^2 = 0
\]
with $a_1,a_2,a_3,a_4,a_6 \in K$.
We often identify an elliptic curve (as well as a line and a curve) with the polynomial in its defining equation.   
We write the set of $K$-rational points of $E$ as $E(K) = \{ P \in \mathbb{P}^2 \mid E(P) = 0 \}$.
We use notations such as $[\alpha:\beta:\gamma]$ to express the projective coordinates for a point in $\mathbb{P}^2$.
Let $O = [0:1:0]$ denote the point at infinity in $E$.
We note that any other rational point of $E$ has non-zero $Z$-coordinate, therefore such a point can be expressed in a way that the $Z$-coordinate is normalized to one.

We omit proofs of the following basic properties, but note that these properties can be proved elementary.
A line $\ell$ in $\mathbb{P}^2$ means the set of solutions $[X:Y:Z]$ for an equation of the form $\ell \colon A X + B Y + C Z = 0$ with non-zero coefficient vector $(A,B,C) \in K^3$ (note that a scalar multiplication to the coefficient vector does not affect the corresponding line).
For any two distinct points $P,Q \in \mathbb{P}^2$, there is a unique line in $\mathbb{P}^2$ passing through $P$ and $Q$; in this paper, we denote it by $\overline{P|Q}$.
On the other hand, by writing the formal derivative of a polynomial $g$ by a variable $t$ as $\partial_t g$, the tangent line of $E$ at a point $P \in E(K)$ is defined to be the line in $\mathbb{P}^2$ with coefficient vector
\[
(T_X(P), T_Y(P), T_Z(P)) \eqdef ((\partial_X E)(P), (\partial_Y E)(P), (\partial_Z E)(P)) \enspace.
\]
(We note that we are considering elliptic curves without singular points, therefore the coefficient vector is always non-zero.)
In this paper, we denote this tangent line by $\overline{P|P}$.
It is easily seen that the line $\overline{P|P}$ indeed passes through $P$.
Concretely, we have
\[
\begin{split}
\partial_X E &= 3 X^2 + 2 a_2 X Z + a_4 Z^2 - a_1 YZ \enspace,\\
\partial_Y E &= - 2 Y Z - a_1 X Z - a_3 Z^2 \enspace,\\
\partial_Z E &= a_2 X^2 + 2 a_4 X Z + 3 a_6 Z^2 - Y^2 - a_1 X Y - 2 a_3 Y Z \enspace.
\end{split}
\]

When $P,Q \in E(K)$ are distinct, it holds either that $E$ and $\overline{P|Q}$ has precisely one more intersection point (over $K$) or that
$P$ or $Q$ (but not both) is the point of tangency of $\overline{P|Q}$ to $E$.
In the former case, we define $P \ast Q$ to be the third intersection point; and in the latter case, we define $P \ast Q$ to be the point of tangency.
On the other hand, when $P \in E(K)$, it holds either that $E$ and $\overline{P|P}$ has precisely one more intersection point (over $K$) or that $E$ and $\overline{P|P}$ intersects at $P$ only.
In the former case, we define $P \ast P$ to be the second intersection point; and in the latter case, we define $P \ast P \eqdef P$.
In particular, the point $O$ lies in the latter case and hence $O \ast O = O$.
Note that $P \ast Q = Q \ast P$ for any $P,Q \in E(K)$ by the symmetry of the definition of $\ast$. 
We define
\[
-P \eqdef P \ast O \mbox{ for } P \in E(K) \enspace.
\]
It is shown that $-O = O$ and
\begin{equation}
\label{eq:negation_formula}
-P = [\alpha: - a_1 \alpha - a_3 - \beta: 1] \mbox{ for } P = [\alpha: \beta: 1] \in E(K) \enspace.
\end{equation}
(We emphasize that an explicit formula for $P \ast Q$ is not used.)
By definition and \eqref{eq:negation_formula},
\begin{equation}
\label{eq:cancellation_property}
(P \ast Q) \ast P = Q \mbox{ and } -(P \ast Q) = (-P) \ast (-Q) \mbox{ for } P,Q \in E(K) \enspace,
\end{equation}
therefore $P \ast (-P) = O$ and $-(-P) = P$ for $P \in E(K)$.
These relations are frequently used in our proof.

We are going to give a proof of the fact that the operator $+$ defined by 
\[
P + Q \eqdef -(P \ast Q) \mbox{ for } P,Q \in E(K)
\]
is associative.
Note that this is a commutative operator.
By \eqref{eq:cancellation_property}, we have
\[
(P + Q) + R
= -( ( -(P \ast Q) ) \ast R)
= ( -( -(P \ast Q) ) ) \ast (-R)
= (P \ast Q) \ast (-R)
\]
and (by switching $P$ and $R$ in the argument above)
\[
P + (Q + R) = (R \ast Q) \ast (-P) \enspace.
\]
Hence it suffices to prove the following property:

\begin{theorem}
\label{thm:rational_point_group}
We have $(P \ast Q) \ast (-R) = (R \ast Q) \ast (-P)$ for any $P,Q,R \in E(K)$.
\end{theorem}

Note that this claim is symmetric with respect to $P$ and $R$.

\section{Proof: Some Obvious Cases}
\label{appendix:sec:proof_thm:rational_point_group__point_coincidence}

First, we discuss some \lq\lq obvious'' cases as follows, where we write \lq\lq LHS'' and \lq\lq RHS'' for the left-hand side and the right-hand side of the equality in the theorem, respectively.
\begin{enumerate}
\item
\label{item:obvious__P_O}
If $P = O$ (or equivalently, $-P = O$), then LHS becomes $(-Q) \ast (-R) = -(Q \ast R)$ and RHS becomes $(R \ast Q) \ast O = -(Q \ast R)$, as desired.
By the aforementioned symmetry in the claim, the case $R = O$ (or equivalently, $-R = O$) is similar.
\item
\label{item:obvious__Q_O}
If $Q = O$, then both LHS and RHS become $(-P) \ast (-R)$, as desired.
Moreover, if $P \ast Q = -P$ or $R \ast Q = -R$, then we have $Q = O$ and hence the claim also holds.
\item
\label{item:obvious__P_R}
If $P = R$ (or equivalently, $-P = -R$), then LHS and RHS become identical, as desired.
Moreover, if $P \ast Q = R \ast Q$, then we have $P = R$ and the claim also holds.
\item
\label{item:obvious__Q_-P}
If $Q = -P$ (or equivalently, $P \ast Q = O$), then LHS becomes $O \ast (-R) = R$ and RHS becomes $(R \ast Q) \ast Q = R$, as desired.
By symmetry, the case $Q = -R$ (or equivalently, $R \ast Q = O$) is similar.
\item
\label{item:obvious__P_-R}
If $P = -R$ (or equivalently, $R = -P$), then LHS becomes $(P \ast Q) \ast P = Q$ and RHS becomes $(R \ast Q) \ast R = Q$, as desired.
\item
\label{item:obvious__P_R*Q}
If $P = R \ast Q$ (or equivalently, $R = P \ast Q$), then LHS becomes $R \ast (-R) = O$ and RHS becomes $P \ast (-P) = O$, as desired.
\item
\label{item:obvious__O_P*Q*-R}
If $O = (P \ast Q) \ast (-R)$, then we have $P \ast Q = O \ast (-R) = R$ and this is reduced to Case \ref{item:obvious__P_R*Q}.
The case $O = (R \ast Q) \ast (-P)$ is similar.
\item
\label{item:obvious__P_P*Q*-R}
If $P = (P \ast Q) \ast (-R)$, then we have $-R = (P \ast Q) \ast P = Q$ and this is reduced to Case \ref{item:obvious__Q_-P}.
The case $R = (R \ast Q) \ast (-P)$ is similar.
\item
\label{item:obvious__P_R*Q*-P}
If $P = (R \ast Q) \ast (-P)$, then we have $R \ast Q = P \ast (-P) = O$ and this is reduced to Case \ref{item:obvious__Q_-P}.
The case $R = (P \ast Q) \ast (-R)$ is similar.
\item
\label{item:obvious__Q_P*Q*-R}
If $Q = (P \ast Q) \ast (-R)$, then we have $-R = (P \ast Q) \ast Q = P$ and this is reduced to Case \ref{item:obvious__P_-R}.
The case $Q = (R \ast Q) \ast (-P)$ is similar.
\end{enumerate}
Hence the claim holds in any of the cases above; this is summarized in Table \ref{tab:addition_on_elliptic_curve__equality_for_points}.
From now, we use the names of the points as in the table.
Then our claim is to prove that $P_9 = P_{10}$, and we may assume without loss of generality that the points $P_i$ and $P_j$ are not equal whenever the corresponding cell in Table \ref{tab:addition_on_elliptic_curve__equality_for_points} is filled with a number.
\begin{table}[t!]
\centering
\caption{Possibilites for coincidence of the points; for any numbered cell, the claim has been verified when the two points coincide}
\label{tab:addition_on_elliptic_curve__equality_for_points}
\begin{tabular}{c|c|c|c|c|c|c|c|c|c|c|}
& $P_1$ & $P_2$ & $P_3$ & $P_4$ & $P_5$ & $P_6$ & $P_7$ & $P_8$ & $P_9$ & $P_{10}$ \\ \hline
$P_1 = O$ & \multicolumn{1}{c|}{} & \ref{item:obvious__P_O} & \ref{item:obvious__P_O} & \ref{item:obvious__P_O} & \ref{item:obvious__P_O} & \ref{item:obvious__Q_O} & \ref{item:obvious__Q_-P} & \ref{item:obvious__Q_-P} & \ref{item:obvious__O_P*Q*-R} & \ref{item:obvious__O_P*Q*-R} \\ \cline{3-11}
$P_2 = P$ & \multicolumn{2}{c|}{} & & \ref{item:obvious__P_R} & \ref{item:obvious__P_-R} & & & \ref{item:obvious__P_R*Q} & \ref{item:obvious__P_P*Q*-R} & \ref{item:obvious__P_R*Q*-P} \\ \cline{4-11}
$P_3 = -P$ & \multicolumn{3}{c|}{} & \ref{item:obvious__P_-R} & \ref{item:obvious__P_R} & \ref{item:obvious__Q_-P} & \ref{item:obvious__Q_O} & & & \\ \cline{5-11}
$P_4 = R$ & \multicolumn{4}{c|}{} & & & \ref{item:obvious__P_R*Q} & & \ref{item:obvious__P_R*Q*-P} & \ref{item:obvious__P_P*Q*-R} \\ \cline{6-11}
$P_5 = -R$ & \multicolumn{5}{c|}{} & \ref{item:obvious__Q_-P} & & \ref{item:obvious__Q_O} & & \\ \cline{7-11}
$P_6 = Q$ & \multicolumn{6}{c|}{} & & & \ref{item:obvious__Q_P*Q*-R} & \ref{item:obvious__Q_P*Q*-R} \\ \cline{8-11}
$P_7 = P \ast Q$ & \multicolumn{7}{c|}{} & \ref{item:obvious__P_R} & & \\ \cline{9-11}
$P_8 = R \ast Q$ & \multicolumn{8}{c|}{} & & \\ \cline{10-11}
$P_9 = (P \ast Q) \ast (-R)$ & \multicolumn{9}{c|}{} & \\ \cline{11-11}
$P_{10} = (R \ast Q) \ast (-P)$ & \multicolumn{10}{c|}{} \\ \hline
\end{tabular}
\end{table}

\section{Proof: Cases with Less Point Coincidence}
\label{appendix:sec:proof_thm:rational_point_group__less_coincidence}

In this section, we prove the following property.

\begin{proposition}
\label{prop:addition_on_E__multiplicity_less_than_three}
If no three points among $P_1,\dots,P_9$ are equal, then $P_9 = P_{10}$.
\end{proposition}

We prepare some notations.
Let $\mathbb{M}$ denote the vector consisting of the monic monomials of degree three in variables $X$, $Y$, and $Z$, defined by
\[
\mathbb{M}
\eqdef (X^3, Y^3, Z^3, X^2 Y, X Y^2, X^2 Z, X Z^2, Y^2 Z, Y Z^2, X Y Z) \enspace.
\]
We use terminology such as \lq\lq Column $Y^2 Z$'' to indicate the corresponding column of the vector, and write e.g., $\mathbb{M}[Y^2 Z]$ to denote the component of the vector at the indicated column (we also use similar terminology and notations for other vectors introduced later).
For $P' \in \mathbb{P}^2$, we denote by $\mathbb{M}(P')$ the vector obtained by substituting the coordinate values of $P'$ (chosen among the uncertainty of scalar multiple) into the variables in $\mathbb{M}$.
For a homogeneous polynomial $F = F(X,Y,Z)$ of degree three, we denote by $c_{F,M}$ the coefficient of a monomial $M$ in $F$.
That is, $F(X,Y,Z) = \sum_{M} c_{F,M} M$ where the index $M$ runs over the ten monomials involved in $\mathbb{M}$.
We denote by $c_F$ the vector consisting of the $c_{F,M}$'s ordered in the same way as $\mathbb{M}$.
Then we have
\begin{equation}
\label{eq:value_of_cubic_polynomial_by_inner_product}
F(P') = \mathbb{M}(P') \cdot {}^t c_F
\end{equation}
(where ${}^t (\cdot)$ denotes the transpose).
On the other hand, for each variable $t \in \{X,Y,Z\}$, let $\mathbb{M}_t$ denote the vector obtained from $\mathbb{M}$ by taking the derivative of each component with respect to the variable $t$.
Concretely,
\[
\begin{split}
\mathbb{M}_X
&\eqdef (3 X^2, 0, 0, 2 X Y, Y^2, 2 X Z, Z^2, 0, 0, Y Z) \enspace,\\
\mathbb{M}_Y
&\eqdef (0, 3 Y^2, 0, X^2, 2 X Y, 0, 0, 2 Y Z, Z^2, X Z) \enspace,\\
\mathbb{M}_Z
&\eqdef (0, 0, 3 Z^2, 0, 0, X^2, 2 X Z, Y^2, 2 Y Z, X Y) \enspace.
\end{split}
\]
Then for $P'$ and $F$ as above, we have
\[
(\partial_t F)(P') = \mathbb{M}_t(P') \cdot {}^t c_F \mbox{ for any } t \in \{X,Y,Z\} \enspace.
\]
Now for any $P' \in E(K)$, the definition of the tangent line $\overline{P'|P'}$ implies:
\begin{equation}
\label{eq:tangent_line_by_inner_product__elliptic_curve}
\mbox{There is a $\lambda \in K$ satisfying }
\begin{pmatrix}
\mathbb{M}_X(P') \\
\mathbb{M}_Y(P') \\
\mathbb{M}_Z(P')
\end{pmatrix}
\cdot {}^t c_E
= \lambda \cdot
\begin{pmatrix}
T_X(P') \\
T_Y(P') \\
T_Z(P')
\end{pmatrix}
\enspace.
\end{equation}
On the other hand, we have the following property.

\begin{lemma}
\label{lem:tangent_line_by_inner_product__three_lines}
Let $\ell_1$, $\ell_2$, and $\ell_3$ be lines in $\mathbb{P}^2$, and we define a homogeneous polynomial of degree three by $F = \ell_1 \ell_2 \ell_3$.
Let $P' = [\alpha': \beta': \gamma'] \in E(K) \setminus \{O\}$, and suppose that $\ell_1(P') = 0$.
Then the following conditions are equivalent.
In this case, we say that $P'$ is a multiple intersection point of $E$ and $F$.
\begin{itemize}
\item
There is a $\lambda \in K$ satisfying
\begin{equation}
\label{eq:lem:tangent_line_by_inner_product__three_lines__condition}
\begin{pmatrix}
\mathbb{M}_X(P') \\
\mathbb{M}_Y(P')
\end{pmatrix}
\cdot {}^t c_F
= \lambda \cdot
\begin{pmatrix}
T_X(P') \\
T_Y(P')
\end{pmatrix}
\enspace.
\end{equation}
\item
Either $\ell_1 = \overline{P'|P'}$, or at least one of $\ell_2$ and $\ell_3$ passes through $P'$.
\end{itemize}
\end{lemma}
\begin{proof}
Let $(A,B,C)$ be a coefficient vector for $\ell_1$.
As $\ell_1(P') = 0$, by the property of formal derivative for product of polynomials, we have
\begin{equation}
\label{eq:lem:tangent_line_by_inner_product__three_lines__relation}
\begin{pmatrix}
\mathbb{M}_X(P') \\
\mathbb{M}_Y(P') \\
\mathbb{M}_Z(P')
\end{pmatrix}
\cdot {}^t c_F
=
\begin{pmatrix}
(\partial_X F)(P') \\
(\partial_Y F)(P') \\
(\partial_Z F)(P')
\end{pmatrix}
= \ell_2(P') \ell_3(P') \cdot
\begin{pmatrix}
A \\
B \\
C
\end{pmatrix}
\enspace.
\end{equation}
Here we note that the condition that at least one of $\ell_2$ and $\ell_3$ passes through $P'$ is equivalent to the condition $\ell_2(P') \ell_3(P') = 0$.
If $\ell_2(P') \ell_3(P') = 0$, then the claim holds obviously (with $\lambda = 0$).
From now, we consider the other case $\ell_2(P') \ell_3(P') \neq 0$.
If $\ell_1 = \overline{P'|P'}$, then $(A,B)$ is a scalar multiple of $(T_X(P'),T_Y(P'))$ by the definition of the tangent line, therefore \eqref{eq:lem:tangent_line_by_inner_product__three_lines__relation} implies the claim.
Conversely, suppose that \eqref{eq:lem:tangent_line_by_inner_product__three_lines__condition} is satisfied.
By the relation $X \cdot \mathbb{M}_X + Y \cdot \mathbb{M}_Y + Z \cdot \mathbb{M}_Z = 3 \cdot \mathbb{M}$, we have
\[
\alpha' \cdot \mathbb{M}_X(P') + \beta' \cdot \mathbb{M}_Y(P') + \gamma' \cdot \mathbb{M}_Z(P')
= 3 \cdot \mathbb{M}(P') \enspace,
\]
therefore the fact $F(P') = 0$ and \eqref{eq:value_of_cubic_polynomial_by_inner_product} imply
\[
\gamma' \cdot \mathbb{M}_Z(P') \cdot {}^t c_F
= - \alpha' \cdot \mathbb{M}_X(P') \cdot {}^t c_F
- \beta' \cdot \mathbb{M}_Y(P') \cdot {}^t c_F \enspace.
\]
On the other hand, as the tangent line $\overline{P'|P'}$ passes through $P'$, we have
\[
T_X(P') \cdot \alpha' + T_Y(P') \cdot \beta' + T_Z(P') \cdot \gamma'
= 0 \enspace.
\]
By these properties together with \eqref{eq:lem:tangent_line_by_inner_product__three_lines__condition} and the fact $\gamma' \neq 0$ (recall that $P' \neq O$), it follows that $\mathbb{M}_Z(P') \cdot {}^t c_F = \lambda \cdot T_Z(P')$.
Moreover, by the current assumption $\ell_2(P') \ell_3(P') \neq 0$, it follows from \eqref{eq:lem:tangent_line_by_inner_product__three_lines__relation} that $(A,B,C)$ is a scalar multiple of $(T_X(P'),T_Y(P'),T_Z(P'))$, therefore we have $\ell_1 = \overline{P'|P'}$.
This completes the proof.
\end{proof}

We define homogeneous polynomials $F_1$ and $F_2$ of degree three by
\[
F_1 \eqdef \overline{P|-P} \cdot \overline{R|Q} \cdot \overline{P \ast Q|-R} \,,\,
F_2 \eqdef \overline{R|-R} \cdot \overline{P|Q} \cdot \overline{R \ast Q|-P} \enspace.
\]
By definition, we have $F_1(P_9) = 0$ and $F_2(P_{10}) = 0$, and also $F_1(P_i) = F_2(P_i) = 0$ for $1 \leq i \leq 8$.
In the following argument, we are going to show the existence of a linear relation for vectors $c_E$, $c_{F_1}$, and $c_{F_2}$, by utilizing the relations $F_1(P_i) = F_2(P_i) = 0$ above and a system of linear equations obtained from \eqref{eq:value_of_cubic_polynomial_by_inner_product}, \eqref{eq:tangent_line_by_inner_product__elliptic_curve}, and \eqref{eq:lem:tangent_line_by_inner_product__three_lines__condition}.
By recalling the assumption due to Table \ref{tab:addition_on_elliptic_curve__equality_for_points} that $P_i \neq O$ for any $i \geq 2$, we set $P_1 = [0: 1: 0]$ and $P_i = [\alpha_i: \beta_i: 1]$ with $\alpha_i,\beta_i \in K$ for $i \geq 2$.

We construct auxiliary sets $I,J \subseteq \{1,2,\dots,8\}$ of indices in the following manner.
First, for $1 \leq i \leq 8$, if $P_i \neq P_j$ for any $1 \leq j \leq 8$ with $j \neq i$, then we add the $i$ to $I$.
On the other hand, for different indices $i,j \in \{1,\dots,8\}$, if $P_i = P_j$, then we add one of these two indices to $I$ and add the other index to $J$.
(Recall the current hypothesis that no three points among $P_1,\dots,P_9$ coincide with each other.)
Now in the latter case, the index among $i$ and $j$ to be added to $I$ is chosen by the following rule.
\begin{enumerate}
\item
If either $i$ or $j$ is in $\{2,4\}$, then we always add this index to $I$.
(Recall from Table \ref{tab:addition_on_elliptic_curve__equality_for_points} that now $P_2 \neq P_4$, therefore this index is uniquely determined.)
\item
Otherwise, if either $i$ or $j$ is in $\{3,5\}$, then we always add this index to $I$.
(Recall from Table \ref{tab:addition_on_elliptic_curve__equality_for_points} that now $P_3 \neq P_5$, therefore this index is uniquely determined.)
\item
Otherwise, if $i,j \in \{6,7,8\}$, then we have $6 \in \{i,j\}$ (recall from Table \ref{tab:addition_on_elliptic_curve__equality_for_points} that now $P_7 \neq P_8$), and we add the index $6$ to $J$ and the other index to $I$.
(In the other cases, we may freely choose any of the two indices.)
\end{enumerate}
By definition, $\{1,\dots,8\}$ is the disjoint union of $I$ and $J$, and we always have $1,2,4 \in I$.
Moreover, as $P_3 \neq P_4$ and $P_5 \neq P_2$ by Table \ref{tab:addition_on_elliptic_curve__equality_for_points}, the conditions $3 \in J$ and $P_2 = P_3$ are equivalent, and the conditions $5 \in J$ and $P_4 = P_5$ are equivalent.

Now for each $i \in I$, we define the vector $v[i] \in K^{10+|J|}$ by adding $|J|$ components with entry $0$ at the end of the vector $\mathbb{M}(P_i) \in K^{10}$.
Now each of the added component can be associated to an element $j$ of $J$; we write \lq\lq Column $P_j$'' and write $v[i][P_j]$ to indicate the component associated to $j \in J$ (we also use similar terminology and notation for the vectors introduced below).
On the other hand, for each $j \in J$, we define the vector $v[j,X]$ (respectively, $v[j,Y]$) in $K^{10+|J|}$ by adding $|J|$ components at the end of the vector $\mathbb{M}_X(P_j) \in K^{10}$ (respectively, $\mathbb{M}_Y(P_j) \in K^{10}$) in a way that its entry at Column $P_j$ becomes $-T_X(P_j)$ (respectively, $-T_Y(P_j)$) and that at Column $P_k$ ($k \in J$, $k \neq j$) becomes $0$.
We have obtained $|I| + 2 |J| = 8 + |J|$ row vectors in $K^{10+|J|}$; we define $H$ to be the matrix over $K$ consisting of the $8 + |J|$ row vectors.
(See Figure \ref{fig:matrix_H} for an example of the structure of matrix $H$ for the case $P_3 = P_4$ and $P_6 = P_8$.)
Then we have the following.

\begin{figure}[t]
\centering
\begin{tabular}{lr|c|cc|} \cline{3-5}
Row $v[1]$ & $\to$ & \hbox to60pt{\hss$\mathbb{M}(P_1)$\hss} & $0$ & $0$ \\ \cline{3-3}
Row $v[2]$ & $\to$ & \hbox to60pt{\hss$\mathbb{M}(P_2)$\hss} & $0$ & $0$ \\ \cline{3-3}
Row $v[3,X]$ & $\to$ & \hbox to60pt{\hss$\mathbb{M}_X(P_3)$\hss} & $-T_X(P_3)$ & $0$ \\ \cline{3-3}
Row $v[3,Y]$ & $\to$ & \hbox to60pt{\hss$\mathbb{M}_Y(P_3)$\hss} & $-T_Y(P_3)$ & $0$ \\ \cline{3-3}
Row $v[4]$ & $\to$ & \hbox to60pt{\hss$\mathbb{M}(P_4)$\hss} & $0$ & $0$ \\ \cline{3-3}
Row $v[5]$ & $\to$ & \hbox to60pt{\hss$\mathbb{M}(P_5)$\hss} & $0$ & $0$ \\ \cline{3-3}
Row $v[6,X]$ & $\to$ & \hbox to60pt{\hss$\mathbb{M}_X(P_6)$\hss} & $0$ & $-T_X(P_6)$ \\ \cline{3-3}
Row $v[6,Y]$ & $\to$ & \hbox to60pt{\hss$\mathbb{M}_Y(P_6)$\hss} & $0$ & $-T_Y(P_6)$ \\ \cline{3-3}
Row $v[7]$ & $\to$ & \hbox to60pt{\hss$\mathbb{M}(P_7)$\hss} & $0$ & $0$ \\ \cline{3-3}
Row $v[8]$ & $\to$ & \hbox to60pt{\hss$\mathbb{M}(P_8)$\hss} & $0$ & $0$ \\ \cline{3-3}\cline{3-5}
\multicolumn{2}{c}{} & \multicolumn{1}{c}{$\underbrace{\hbox to60pt{}}_{\mbox{$10$ columns}}$} & \multicolumn{1}{c}{$\underbrace{\phantom{-T_X(P_3)}}_{\mbox{Column $P_3$}}$} & \multicolumn{1}{c}{$\underbrace{\phantom{-T_X(P_6)}}_{\mbox{Column $P_6$}}$} \\
\end{tabular}
\caption{An example of the matrix $H$, for the case where $P_3 = P_4$ and $P_6 = P_8$; hence $I = \{1,2,4,5,7,8\}$ and $J = \{3,6\}$}
\label{fig:matrix_H}
\end{figure}

\begin{lemma}
\label{lem:addition_on_E__relation_of_auxiliary_matrix}
There are vectors $w_E,w_{F_1},w_{F_2} \in K^{|J|}$ satisfying that $H \cdot {}^t (c_E \mid\mid w_E) = \vec{0}$ and $H \cdot {}^t (c_{F_b} \mid\mid w_{F_b}) = \vec{0}$ for $b = 1,2$, where \lq\lq $\mid\mid$'' denotes the concatenation of vectors.
\end{lemma}
\begin{proof}
First, for $i \in I$, we have $E(P_i) = F_1(P_i) = F_2(P_i) = 0$ as $i \leq 8$.
Now by \eqref{eq:value_of_cubic_polynomial_by_inner_product} and the definition of $v[i]$, we have (regardless of the vector $w_E$)
\[
v[i] \cdot {}^t (c_E \mid\mid w_E)
= \mathbb{M}(P_i) \cdot {}^t c_E
= E(P_i)
= 0
\]
and similarly $v[i] \cdot {}^t (c_{F_b} \mid\mid w_{F_b}) = 0$ for $b = 1,2$.

Secondly, for $j \in J$, we can take an index $i \in I$ with $P_j = P_i$.
We take $\lambda \in K$ as in \eqref{eq:tangent_line_by_inner_product__elliptic_curve} for $P' = P_j$ and write it as $\lambda_{E,j}$.
By setting $w_E[P_j] = \lambda_{E,j}$, the definition of $v[j,X]$ implies (regardless of the other columns of $w_E$)
\[
v[j,X] \cdot {}^t (c_E \mid\mid w_E)
= \mathbb{M}_X(P_j) \cdot {}^t c_E - \lambda_{E,j} \cdot T_X(P_j)
= 0
\]
and similarly $v[j,Y] \cdot {}^t (c_E \mid\mid w_E) = 0$.
On the other hand, for $b = 1,2$, the definition of $F_b$ implies that $P_j = P_i$ is a multiple intersection point of $E$ and $F_b$ in the sense of Lemma \ref{lem:tangent_line_by_inner_product__three_lines} (see the latter condition in that lemma).
Now Lemma \ref{lem:tangent_line_by_inner_product__three_lines} implies that there is a $\lambda \in K$ satisfying \eqref{eq:lem:tangent_line_by_inner_product__three_lines__condition} for $P' = P_j$.
We write the $\lambda$ as $\lambda_{F_b,j}$.
By setting $w_{F_b}[P_j] = \lambda_{F_b,j}$, the definition of $v[j,X]$ implies (regardless of the other columns of $w_{F_b}$)
\[
v[j,X] \cdot {}^t (c_{F_b} \mid\mid w_{F_b})
= \mathbb{M}_X(P_j) \cdot {}^t c_{F_b} - \lambda_{F_b,j} \cdot T_X(P_j)
= 0
\]
and similarly $v[j,Y] \cdot {}^t (c_{F_b} \mid\mid w_{F_b}) = 0$.
By these arguments, the claim holds by choosing the components of vectors $w_E$, $w_{F_1}$, and $w_{F_2}$ as above.
\end{proof}

By Lemma \ref{lem:addition_on_E__auxiliary_matrix_is_row_full-rank} given later, $H$ has rank $8 + |J|$, therefore the kernel of $H$ has dimension two.
On the other hand, the vectors $c_E$ and $c_{F_1}$ are linearly independent.
Indeed, if $c_E$ were a scalar multiple of $c_{F_1}$, then the properties $c_{E,X^3},c_{E,Y^2} \neq 0$ and $c_{E,Y^3} = 0$ would imply that for the three degree-$1$ factors of $F_1$, all of the coefficients of $X$ must be non-zero and precisely two of the coefficients of $Y$ must be non-zero.
However, now we have $c_{F_1,X Y^2} \neq 0$, contradicting the fact $c_{E,X Y^2} = 0$.
Now the linear independence of $c_E$ and $c_{F_1}$ implies that the vectors $\widehat{c}_E \eqdef (c_E \mid\mid w_E)$ and $\widehat{c}_{F_1} \eqdef (c_{F_1} \mid\mid w_{F_1})$ in Lemma \ref{lem:addition_on_E__relation_of_auxiliary_matrix} are also linearly independent.
Hence, $\widehat{c}_E$ and $\widehat{c}_{F_1}$ form a basis of the kernel of $H$, therefore the vector $\widehat{c}_{F_2} \eqdef (c_{F_2} \mid\mid w_{F_2})$ lying in the kernel of $H$ as well must be a linear combination of $\widehat{c}_E$ and $\widehat{c}_{F_1}$.
Now there are the following two cases.
\begin{itemize}
\item
Suppose that $P_9 \neq P_k$ for any $1 \leq k \leq 8$.
As $E(P_9) = F_1(P_9) = 0$ by definition, we have $\mathbb{M}(P_9) \cdot {}^t c_E = \mathbb{M}(P_9) \cdot {}^t c_{F_1} = 0$, while $c_{F_2}$ is a linear combination of $c_E$ and $c_{F_1}$ as discussed above.
Hence it follows that $\mathbb{M}(P_9) \cdot {}^t c_{F_2} = 0$, therefore $F_2(P_9) = 0$.
This means that at least one of the three lines $\overline{R|-R}$, $\overline{P|Q}$, and $\overline{R \ast Q|-P}$ forming $F_2$ must pass through $P_9$.
Now by the assumption that $P_9$ is different from $P_1,\dots,P_8$, the former two lines cannot pass through $P_9$, therefore $\overline{R \ast Q|-P}$ must pass through $P_9$.
Moreover, as $P_9 \neq P_3 = -P$ and $P_9 \neq P_7 = R \ast Q$ by the assumption, we must have $P_9 = (R \ast Q) \ast (-P) = P_{10}$, as desired.
\item
Suppose that $P_9 = P_i$ with $1 \leq i \leq 8$.
Then the same argument as above implies that $F_2(P_9) = 0$.
Moreover, by \eqref{eq:tangent_line_by_inner_product__elliptic_curve}, the vector $(\mathbb{M}_X(P_9) \cdot {}^t c_E, \mathbb{M}_Y(P_9) \cdot {}^t c_E)$ is a scalar multiple of $(T_X(P_9),T_Y(P_9))$.
On the other hand, by the definition of $F_1$ and the assumption $P_9 = P_i$, it follows that $P_9$ is a multiple intersection point of $E$ and $F_1$, therefore Lemma \ref{lem:tangent_line_by_inner_product__three_lines} implies that $(\mathbb{M}_X(P_9) \cdot {}^t c_{F_1}, \mathbb{M}_Y(P_9) \cdot {}^t c_{F_1})$ is also a scalar multiple of $(T_X(P_9),T_Y(P_9))$.
Now as $c_{F_2}$ is a linear combination of $c_E$ and $c_{F_1}$ as discussed above, it follows that $(\mathbb{M}_X(P_9) \cdot {}^t c_{F_2}, \mathbb{M}_Y(P_9) \cdot {}^t c_{F_2})$ is also a scalar multiple of $(T_X(P_9),T_Y(P_9))$.
Therefore, by Lemma \ref{lem:tangent_line_by_inner_product__three_lines} again, $P_9 = P_i$ is a multiple intersection point of $E$ and $F_2$.
Now $P_i$ must coincide with some of $P_1,\dots,P_8,P_{10}$, while no three points among $P_1,\dots,P_9$ coincide by the current hypothesis, therefore we must have $P_9 = P_i = P_{10}$, as desired.
\end{itemize}
Hence we have proved Proposition \ref{prop:addition_on_E__multiplicity_less_than_three} by assuming Lemma \ref{lem:addition_on_E__auxiliary_matrix_is_row_full-rank} below.
We also note that, as the claim of Theorem \ref{thm:rational_point_group} is symmetric with respect to $P$ and $R$, and as changing the roles of $P$ and $R$ will switch $P_9$ and $P_{10}$, it follows that the claim also holds (by assuming Lemma \ref{lem:addition_on_E__auxiliary_matrix_is_row_full-rank} below) when no three points among $P_1,\dots,P_8,P_{10}$ coincide.

To conclude this section, we prove the following postponed lemma.

\begin{lemma}
\label{lem:addition_on_E__auxiliary_matrix_is_row_full-rank}
The matrix $H$ defined above has rank $8 + |J|$.
\end{lemma}
\begin{proof}
It suffices to prove that the square submatrix of size $8 + |J|$ obtained from $H$ by removing Column $X^3$ and Column $X Y^2$ is invertible.
We are going to reduce this matrix and decrease the matrix size by elementary row (or sometimes column) transformations.
Here we frequently use the equality $E(P_i) = 0$ for each $i$ and the following consequences of the relations $P_3 = -P_2$ and $P_5 = -P_4$; we have $\alpha_3 = \alpha_2$, $\alpha_5 = \alpha_4$, and
\[
\beta_k + \beta_{k+1} = - a_1 \alpha_k - a_3 \mbox{ and }
\beta_k \beta_{k+1} = - (\alpha_k^3 + a_2 \alpha_k^2 + a_4 \alpha_k + a_6)
\mbox{ for } k = 2,4 \enspace.
\]
First, by observing that $v[1][Y^3] = 1$ and the other columns of $v[1]$ are zero, we remove Column $Y^3$ from the matrix by using reduction by $v[1]$ and then remove the row $v[1]$.

Next, we remove Column $Z^3$ of vectors other than $v[2]$.
For $j \in J$, Columns $Z^3$ of $v[j,X]$ and $v[j,Y]$ are already zero.
For $i = 2,4 \in I$, we have the two cases.
\begin{itemize}
\item
When $i + 1 \in I$, we reduce $v[i+1]$ by $v[i]$.
Now by the definition of $I$ and the fact $i + 1 \in I$, it follows that $P_{i+1} \neq P_i$, while $\alpha_{i+1} = \alpha_i$, therefore $\beta_{i+1} \neq \beta_i$.
Based on this, the vector after reduction can be divided by $\beta_{i+1} - \beta_i \neq 0$ to yield the following vector (we omit the additional columns associated to indices in $J$ as those components are not changed; we also do similarly in the following).
\[
v[i+1] \leftarrow
(\underbrace{\alpha_i^2}_{X^2 Y},
\underbrace{0}_{X^2 Z},
\underbrace{0}_{X Z^2},
\underbrace{\beta_{i+1} + \beta_i}_{Y^2 Z},
\underbrace{1}_{Y Z^2},
\underbrace{\alpha_i}_{X Y Z}) \enspace.
\]
\item
When $i + 1 \in J$, the definition of $I$ and $J$ implies that $P_{i+1} = P_2$ or $P_{i+1} = P_4$; by Table \ref{tab:addition_on_elliptic_curve__equality_for_points}, we in fact have $P_{i+1} = P_i$.
Now we have $- P_i = P_{i+1} = P_i$, therefore $T_Y(P_{i+1}) = 0$; and as $P_i = P_{i+1} \neq O$, we have $T_X(P_{i+1}) \neq 0$.
This implies that currently Column $P_{i+1}$ of each vector other than $v[i+1,X]$ is zero, while $v[i+1,X][P_{i+1}] = -T_X(P_{i+1}) \neq 0$.
Therefore, we can remove the row $v[i+1,X]$ by elementary column transformations using Column $P_{i+1}$, and then remove the Column $P_{i+1}$.
By this operation, $v[i+1,Y]$ becomes
\[
v[i+1,Y] \leftarrow
(\underbrace{\alpha_{i+1}^2}_{X^2 Y},
\underbrace{0}_{X^2 Z},
\underbrace{0}_{X Z^2},
\underbrace{2 \beta_{i+1}}_{Y^2 Z},
\underbrace{1}_{Y Z^2},
\underbrace{\alpha_{i+1}}_{X Y Z})
\]
(note that the omitted components are all zero as Column $P_{i+1}$ has been removed), which is (as now $P_{i+1} = P_i$) equal to the vector $v[i+1]$ obtained in the previous case $i + 1 \in I$.
Based on this, we rewrite the $v[i+1,Y]$ as $v[i+1]$ and move $i + 1$ from $J$ to $I$, which unifies the argument to the previous case where $i + 1 \in I$.
\end{itemize}
By the latter argument above, we may assume without loss of generality that $\{1,2,3,4,5\} \subseteq I$ and $J \subseteq \{6,7,8\}$.
Now we reduce the remaining rows $v[i]$ with $i \in I \setminus \{1,2,3,5\}$ by $v[2]$ (and then remove the row $v[2]$) to obtain
\[
v[i] \leftarrow
(\underbrace{\alpha_i^2 \beta_i - \alpha_2^2 \beta_2}_{X^2 Y},
\underbrace{\alpha_i^2 - \alpha_2^2}_{X^2 Z},
\underbrace{\alpha_i - \alpha_2}_{X Z^2},
\underbrace{\beta_i^2 - \beta_2^2}_{Y^2 Z},
\underbrace{\beta_i - \beta_2}_{Y Z^2},
\underbrace{\alpha_i \beta_i - \alpha_2 \beta_2}_{X Y Z}) \enspace.
\]

Next, we remove Column $Y Z^2$ by reduction using $v[3]$ (and then remove the row $v[3]$).
For $i \in I \setminus \{1,2,3,5\}$, the reduction yields
\[
v[i] \leftarrow
(\underbrace{(\alpha_i^2 - \alpha_2^2) \beta_i}_{X^2 Y},
\underbrace{\alpha_i^2 - \alpha_2^2}_{X^2 Z},
\underbrace{\alpha_i - \alpha_2}_{X Z^2},
\underbrace{(\beta_i - \beta_2)(\beta_i - \beta_3)}_{Y^2 Z},
\underbrace{(\alpha_i - \alpha_2) \beta_i}_{X Y Z}) \enspace.
\]
Now its Column $Y^2 Z$ is equal to $\alpha_i^3 - \alpha_2^3 + a_2 (\alpha_i^2 - \alpha_2^2) + a_4 (\alpha_i - \alpha_2) - a_1 (\alpha_i - \alpha_2) \beta_i$.
Moreover, if $i = 4$ then we have $P_i \neq P_2,P_3$ by Table \ref{tab:addition_on_elliptic_curve__equality_for_points}; while if $i \geq 6$ then we also have $P_i \neq P_2,P_3$ by the definition of $I$.
Hence $\alpha_i \neq \alpha_2$ in any case; now the vector above can be divided by $\alpha_i - \alpha_2 \neq 0$ to yield
\[
v[i] \leftarrow
(\underbrace{(\alpha_i + \alpha_2) \beta_i}_{X^2 Y},
\underbrace{\alpha_i + \alpha_2}_{X^2 Z},
\underbrace{1}_{X Z^2},
\underbrace{\alpha_i^2 + \alpha_i \alpha_2 + \alpha_2^2 + a_2 (\alpha_i + \alpha_2) + a_4 - a_1 \beta_i}_{Y^2 Z},
\underbrace{\beta_i}_{X Y Z}) \enspace.
\]
For the reduction of $v[5]$, note that $P_2 \neq \pm P_4$ by Table \ref{tab:addition_on_elliptic_curve__equality_for_points}, therefore $\alpha_2 \neq \alpha_4$.
Based on this, the resulting vector of the reduction can be divided by $\alpha_4 - \alpha_2 \neq 0$ to yield
\[
v[5] \leftarrow
(\underbrace{\alpha_4 + \alpha_2}_{X^2 Y},
\underbrace{0}_{X^2 Z},
\underbrace{0}_{X Z^2},
\underbrace{- a_1}_{Y^2 Z},
\underbrace{1}_{X Y Z}) \enspace.
\]
Moreover, for $j \in J$, Column $Y Z^2$ of $v[j,X]$ is already zero, while $v[j,Y]$ is reduced as follows, where we used the relation $T_Y(P_j) = - 2 \beta_j - a_1 \alpha_j - a_3$:
\[
v[j,Y] \leftarrow
(\underbrace{\alpha_j^2 - \alpha_2^2}_{X^2 Y},
\underbrace{0}_{X^2 Z},
\underbrace{0}_{X Z^2},
\underbrace{-T_Y(P_j) - a_1 (\alpha_j - \alpha_2)}_{Y^2 Z},
\underbrace{\alpha_j - \alpha_2}_{X Y Z}) \enspace.
\]

Next, by reduction using $v[5]$, we remove Column $X Y Z$ of the remaining vectors $v[i]$ with $i \in I \setminus \{1,2,3,5\}$ and $v[j,X]$ and $v[j,Y]$ with $j \in J$ (and then remove $v[5]$):
\[
\begin{split}
v[i] \leftarrow{}&
(\underbrace{(\alpha_i - \alpha_4) \beta_i}_{X^2 Y},
\underbrace{\alpha_i + \alpha_2}_{X^2 Z},
\underbrace{1}_{X Z^2},
\underbrace{\alpha_i^2 + \alpha_i \alpha_2 + \alpha_2^2 + a_2 (\alpha_i + \alpha_2) + a_4}_{Y^2 Z}) \enspace,\\
v[j,X] \leftarrow{}&
(\underbrace{(2 \alpha_j - \alpha_4 - \alpha_2) \beta_j}_{X^2 Y},
\underbrace{2 \alpha_j}_{X^2 Z},
\underbrace{1}_{X Z^2},
\underbrace{a_1 \beta_j}_{Y^2 Z}) \enspace,\\
v[j,Y] \leftarrow{}&
(\underbrace{(\alpha_j - \alpha_4) (\alpha_j - \alpha_2)}_{X^2 Y},
\underbrace{0}_{X^2 Z},
\underbrace{0}_{X Z^2},
\underbrace{- T_Y(P_j)}_{Y^2 Z}) \enspace.
\end{split}
\]

Next, by reduction using $v[4]$, we remove Column $X Z^2$ of the remaining vectors (and then remove $v[4]$).
For $i \in I \setminus \{1,2,3,4,5\}$, the definition of $I$ implies that $P_i \neq P_4,P_5$, therefore $\alpha_i \neq \alpha_4$.
Based on this, the resulting vector of the reduction can be divided by $\alpha_i - \alpha_4 \neq 0$ to yield
\[
v[i] \leftarrow
(\underbrace{\beta_i}_{X^2 Y},
\underbrace{1}_{X^2 Z},
\underbrace{\alpha_i + \alpha_4 + \alpha_2 + a_2}_{Y^2 Z}) \enspace.
\]
For $j \in J$, Column $X Z^2$ of $v[j,Y]$ is already zero, while $v[j,X]$ is reduced as follows, where we used the relation $T_X(P_j) = 3 \alpha_j^2 + 2 a_2 \alpha_j + a_4 - a_1 \beta_j$:
\[
\begin{split}
v[j,X] \leftarrow{}&
(\underbrace{(2 \alpha_j - \alpha_4 - \alpha_2) \beta_j}_{X^2 Y},
\underbrace{2 \alpha_j - \alpha_4 - \alpha_2}_{X^2 Z}, \\
&\quad
\underbrace{- T_X(P_j) + a_2 (2 \alpha_j - \alpha_4 - \alpha_2) + 3 \alpha_j^2 - (\alpha_4^2 + \alpha_4 \alpha_2 + \alpha_2^2)}_{Y^2 Z}) \enspace.
\end{split}
\]

Now the matrix size has become $3 + |J|$.
For $j \in J$, we have $v[j,X][P_j] = -T_X(P_j)$ and $v[j,Y][P_j] = -T_Y(P_j)$, while Column $P_j$ of the remaining vectors are zero.
Based on this, by subtracting Column $P_j$ from Column $Y^2 Z$, the vectors except for $v[j,X]$ and $v[j,Y]$ are not changed, while $v[j,X]$ and $v[j,Y]$ become
\[
\begin{split}
v[j,X] \leftarrow{}&
(\underbrace{(2 \alpha_j - \alpha_4 - \alpha_2) \beta_j}_{X^2 Y},
\underbrace{2 \alpha_j - \alpha_4 - \alpha_2}_{X^2 Z}, \\
&\quad
\underbrace{a_2 (2 \alpha_j - \alpha_4 - \alpha_2) + 3 \alpha_j^2 - (\alpha_4^2 + \alpha_4 \alpha_2 + \alpha_2^2)}_{Y^2 Z}) \enspace,\\
v[j,Y] \leftarrow{}&
(\underbrace{(\alpha_j - \alpha_4) (\alpha_j - \alpha_2)}_{X^2 Y},
\underbrace{0}_{X^2 Z},
\underbrace{0}_{Y^2 Z}) \enspace.
\end{split}
\]
Now for $k = 2,4$, if some indices $j \in J$ and $k' \in \{k,k+1\}$ satisfy $P_j = P_{k'}$, then by the hypothesis that no three points among $P_1,\dots,P_9$ coincide, it follows that $P_k \neq P_{k+1}$, therefore $T_Y(P_k) \neq 0$ and $T_Y(P_{k+1}) \neq 0$.
This implies that $\alpha_j = \alpha_k$ and $T_Y(P_j) \neq 0$, and now $v[j,Y][P_j] \neq 0$ and the other columns of $v[j,Y]$ are zero.
Hence an elementary row transformation using $v[j,Y]$ can remove Column $P_j$ in $v[j,X]$, and by dividing the resulting vector by $\pm (\alpha_2 - \alpha_4) \neq 0$ we obtain
\[
v[j,X] \leftarrow
(\underbrace{\beta_j}_{X^2 Y},
\underbrace{1}_{X^2 Z},
\underbrace{\alpha_j + \alpha_4 + \alpha_2 + a_2}_{Y^2 Z}) \enspace.
\]
On the other hand, for $j \in J$ and $i \in \{6,7,8\} \setminus \{j\}$, if $P_j = P_i$, then we have $i \in I$ by the definition of $I$, and now the reduction of $v[j,X]$ by $v[i]$ yields
\[
v[j,X] \leftarrow
(\underbrace{0}_{X^2 Y},
\underbrace{0}_{X^2 Z},
\underbrace{(\alpha_j - \alpha_4) (\alpha_j - \alpha_2)}_{Y^2 Z}) \enspace.
\]
Based on these arguments, we perform a case-by-case analysis.
Here we set $I' = I \setminus \{1,2,3,4,5\}$ and put $J' = \{j \in J \mid P_i = P_j \mbox{ for some } i \in I' \setminus \{j\} \}$.
Note that now $J \subseteq \{6,7,8\}$ and $|J'| \leq 1$.
\begin{itemize}
\item
When $J' = \emptyset$, the argument above implies that regardless of which of $I$ and $J$ each $i \in \{6,7,8\}$ belongs to, the current matrix is (by reordering the rows) as in the left-hand side of the expression below.
This is changed to the right-hand side by elementary row transformations for removing the second column:
\[
\begin{pmatrix}
\beta_6 \ \ & 1 \ \ & \alpha_6 + \alpha_4 + \alpha_2 + a_2 \\
\beta_7 \ \ & 1 \ \ & \alpha_7 + \alpha_4 + \alpha_2 + a_2 \\
\beta_8 \ \ & 1 \ \ & \alpha_8 + \alpha_4 + \alpha_2 + a_2
\end{pmatrix}
\to
\begin{pmatrix}
\beta_7 - \beta_6 \ \ & \alpha_7 - \alpha_6 \\
\beta_8 - \beta_6 \ \ & \alpha_8 - \alpha_6
\end{pmatrix}
\enspace.
\]
As $J' = \emptyset$, the three points $P_6 = Q$, $P_7 = P \ast Q$, and $P_8 = R \ast Q$ are all distinct.
If these three points are collinear, then we must have $P_6 \ast P_7 = P_8$, which implies that $P = R \ast Q$ and contradicts Table \ref{tab:addition_on_elliptic_curve__equality_for_points}.
Hence $P_6$, $P_7$, and $P_8$ are not collinear, which implies that the final matrix above is invertible, as desired.
\item
When $J' = \{j\}$ with some $j$, we have $j = 6$ as $P_7 \neq P_8$ by Table \ref{tab:addition_on_elliptic_curve__equality_for_points}.
Choose $i$ and $k$ in a way that $\{7,8\} = \{i,k\}$ and $P_6 = P_k$.
Then we have $P_i \neq P_6$ by the hypothesis that no three points among $P_1,\dots,P_9$ coincide.
Hence we can apply the argument above to show that the current matrix is (by reordering the rows) as follows:
\[
\begin{pmatrix}
\beta_6 \ \ & 1 \ \ & \alpha_6 + \alpha_4 + \alpha_2 + a_2 \ \ & 0 \\
\beta_i \ \ & 1 \ \ & \alpha_i + \alpha_4 + \alpha_2 + a_2 \ \ & 0 \\
0 \ \ & 0 \ \ & (\alpha_6 - \alpha_4) (\alpha_6 - \alpha_2) \ \ & -T_X(P_6) \\
(\alpha_6 - \alpha_4) (\alpha_6 - \alpha_2) \ \ & 0 \ \ & 0 \ \ & -T_Y(P_6)
\end{pmatrix}
\enspace.
\]
By elementary row transformations to remove the second column, the result is
\[
\begin{pmatrix}
\beta_i - \beta_6 \ \ & \alpha_i - \alpha_6 \ \ & 0 \\
0 \ \ & (\alpha_6 - \alpha_4) (\alpha_6 - \alpha_2) \ \ & -T_X(P_6) \\
(\alpha_6 - \alpha_4) (\alpha_6 - \alpha_2) \ \ & 0 \ \ & -T_Y(P_6)
\end{pmatrix}
\]
with determinant $- (\alpha_6 - \alpha_4) (\alpha_6 - \alpha_2) \cdot ( (\beta_i - \beta_6) T_Y(P_6) + (\alpha_i - \alpha_6) T_X(P_6) )$.
Assume for the contrary that this value is zero.
As $6 \in J'$, $P_6$ must be different from $P_2,\dots,P_5$, therefore $\alpha_6 \neq \alpha_2,\alpha_4$.
Hence by dividing the determinant by $\alpha_6 - \alpha_4$ and $\alpha_6 - \alpha_2$, we have $(\beta_i - \beta_6) T_Y(P_6) + (\alpha_i - \alpha_6) T_X(P_6) = 0$, therefore
\[
T_X(P_6) \cdot \alpha_i + T_Y(P_6) \cdot \beta_i + T_Z(P_6)
= T_X(P_6) \cdot \alpha_6 + T_Y(P_6) \cdot \beta_6 + T_Z(P_6)
= 0 \enspace.
\]
This means that the point $P_i$ is on the line $\overline{P_6|P_6} = \overline{Q|Q}$.
Now if $(k,i) = (7,8)$, then the relation $Q = P_6 = P_7 = P \ast Q$ here implies that $P \neq Q$ is also a point on $\overline{Q|Q}$, therefore we have $P_8 = R \ast Q = P$, contradicting Table \ref{tab:addition_on_elliptic_curve__equality_for_points}.
Similarly, if $(k,i) = (8,7)$, then the relation $Q = P_6 = P_8 = R \ast Q$ here implies that $R \neq Q$ is also a point on $\overline{Q|Q}$, therefore we have $P_7 = P \ast Q = R$, contradicting Table \ref{tab:addition_on_elliptic_curve__equality_for_points} again.
Hence the final matrix above is invertible, as desired.
\end{itemize}
This concludes the proof of Lemma \ref{lem:addition_on_E__auxiliary_matrix_is_row_full-rank}.
\end{proof}

\section{Proof: The Remaining Cases}
\label{sec:many_coincidence}

By the argument above, it has been proved that the claim $P_9 = P_{10}$ holds whenever either no three points among $P_1,\dots,P_8,P_9$ coincide, or no three points among $P_1,\dots,P_8,P_{10}$ coincide.
From now, we consider the remaining case.
By Table \ref{tab:addition_on_elliptic_curve__equality_for_points}, the possibilities for some three points among $P_1,\dots,P_8,P_9$ or among $P_1,\dots,P_8,P_{10}$ being equal are only the following six cases: $P_2 = P_6 = P_7$; $P_3 = P_8 = P_9$; $P_3 = P_8 = P_{10}$; $P_4 = P_6 = P_8$; $P_5 = P_7 = P_9$; and $P_5 = P_7 = P_{10}$.

\begin{lemma}
\label{lem:addition_on_E__3-8-9}
If $P_3 = P_8 = P_9$ or $P_5 = P_7 = P_{10}$, then $P_9 = P_{10}$.
\end{lemma}
\begin{proof}
As the case $P_5 = P_7 = P_{10}$ is obtained from $P_3 = P_8 = P_9$ by exchanging $P$ and $R$, it suffices by the symmetry of the claim to consider the case $P_3 = P_8 = P_9$, i.e., $-P = R \ast Q = (P \ast Q) \ast (-R)$.
By the argument above, the claim holds if no three points among $P_1,\dots,P_8,P_{10}$ coincide.
Therefore, it suffices to consider the case where some three points among $P_1,\dots,P_8,P_{10}$ coincide; such possibilities consistent (by Table \ref{tab:addition_on_elliptic_curve__equality_for_points}) with the hypothesis $P_3 = P_8 = P_9$ are: $P_2 = P_6 = P_7$; $P_3 = P_8 = P_{10}$; and $P_5 = P_7 = P_{10}$.
In the second case, we have $P_9 = P_8 = P_{10}$.
In the first case, i.e., $P = Q = P \ast Q$, we have $-P = P_3 = P_8 = R \ast Q = R \ast P$, therefore $R = (-P) \ast P = O$, contradicting Table \ref{tab:addition_on_elliptic_curve__equality_for_points}.
In the third case, i.e., $-R = P \ast Q = (R \ast Q) \ast (-P)$, we have $-P = P_3 = P_9 = (P \ast Q) \ast (-R) = (-R) \ast (-R)$ and $-R = (R \ast Q) \ast (-P) = (-P) \ast (-P)$, therefore $P_9 = P_3 = -P = (-P) \ast (-R) = (-R) = P_5 = P_{10}$.
Hence the claim holds.
\end{proof}

\begin{lemma}
\label{lem:addition_on_E__3-8-10_not_5-7-9}
If $P_3 = P_8 = P_{10}$ holds and $P_5 = P_7 = P_9$ does not hold, then $P_9 = P_{10}$.
\end{lemma}
\begin{proof}
Similarly to the proof of Lemma \ref{lem:addition_on_E__3-8-9}, it suffices to consider the case where some three points among $P_1,\dots,P_8,P_9$ coincide; such possibilities consistent with the hypothesis $P_3 = P_8 = P_{10}$ (except for $P_5 = P_7 = P_9$) are: $P_2 = P_6 = P_7$; and $P_3 = P_8 = P_9$.
In the former case, we have $P = P_2 = P_6 = Q$, while $-P = P_3 = P_8 = R \ast Q$, therefore $R = (-P) \ast Q = (-P) \ast P = O$, contradicting Table \ref{tab:addition_on_elliptic_curve__equality_for_points}.
Hence we have $P_3 = P_8 = P_9$, therefore $P_9 = P_8 = P_{10}$, as desired.
\end{proof}

\begin{lemma}
\label{lem:addition_on_E__2-6-7_4-5}
If $P_2 = P_6 = P_7$ and $P_4 = P_5$, then $P_9 = P_{10}$.
\end{lemma}
\begin{proof}
By the hypothesis, we have $P \ast P = P = Q$.
We put $P' \eqdef -P$, $Q' \eqdef P \ast R$, and $R' \eqdef R$, and from these points we define the points $P'_i$ in the same way as the points $P_i$.
Then we have
\[
P'_3 = -P' = P \,,\,
P'_8 = R' \ast Q' = R \ast (P \ast R) = P \,,\,
P'_{10} = P'_8 \ast P'_3 = P \ast P = P \enspace,
\]
therefore $P'_3 = P'_8 = P'_{10}$.
Now if $P'_5 = P'_7 = P'_9$, then, as $P'_5 = -R' = -R$ and $P'_9 = P'_7 \ast P'_5$ by definition, it follows that $(-R) \ast (-R) = -R$; while $-R = P_5 = P_4 = R$ by the hypothesis.
This implies that $-R = (-R) \ast (-R) = R \ast (-R) = O$, contradicting Table \ref{tab:addition_on_elliptic_curve__equality_for_points}.
Hence $P'_5 = P'_7 = P'_9$ does not hold, therefore Lemma \ref{lem:addition_on_E__3-8-10_not_5-7-9} applied to the points $P'_i$ implies that $P'_9 = P'_{10}$ and
\[
((-P) \ast (P \ast R)) \ast (-R)
= P'_9
= P'_{10}
= (R \ast (P \ast R)) \ast P
= P \ast P
= P \enspace.
\]
Hence we have $(-P) \ast (P \ast R) = P \ast (-R)$ and
\[
P_9
= (P \ast Q) \ast (-R)
= P \ast (-R)
= (-P) \ast (P \ast R)
= (R \ast Q) \ast (-P)
= P_{10}
\]
as desired.
This completes the proof.
\end{proof}

\begin{lemma}
\label{lem:addition_on_E__2-6-7}
If $P_2 = P_6 = P_7$ or $P_4 = P_6 = P_8$, then $P_9 = P_{10}$.
\end{lemma}
\begin{proof}
By the symmetry of the claim with respect to $P$ and $R$, it suffices to consider the case $P_2 = P_6 = P_7$, i.e., $P = Q = P \ast Q$.
By defining the points $P'_i$ in the same way as the proof of Lemma \ref{lem:addition_on_E__2-6-7_4-5}, the claim holds similarly when $P'_5 = P'_7 = P'_9$ does not hold.
We consider the other case where $P'_5 = P'_7 = P'_9$ holds.
We have $-R = P'_5 = P'_7 = (-P) \ast (P \ast R)$, therefore $P \ast R = (-P) \ast (-R) = -(P \ast R)$.
We put $P'' = Q'' \eqdef -P$ and $R'' \eqdef P \ast R$, and from these points we define the points $P''_i$ in the same way as the points $P_i$.
Then we have $P''_7 = P'' \ast Q'' = (-P) \ast (-P) = -(P \ast P) = -P = P''_2 = P''_6$ and $P''_4 = R'' = -R'' = P''_5$.
By Lemma \ref{lem:addition_on_E__2-6-7_4-5} applied to the $P''_i$, we have $P''_9 = P''_{10}$.
Now
\[
P''_9
= ((-P) \ast (-P)) \ast ((-P) \ast (-R))
= (-P) \ast ((-P) \ast (-R))
= -R
\]
and $P''_{10} = ((P \ast R) \ast (-P)) \ast P$, therefore the fact $P''_9 = P''_{10}$ implies
\[
P_9
= (P \ast Q) \ast (-R)
= P \ast (-R)
= (P \ast R) \ast (-P)
= (R \ast Q) \ast (-P)
= P_{10}
\]
as desired.
This completes the proof.
\end{proof}

\begin{lemma}
\label{lem:addition_on_E__3-8-10}
If $P_3 = P_8 = P_{10}$ or $P_5 = P_7 = P_9$, then $P_9 = P_{10}$.
\end{lemma}
\begin{proof}
By symmetry, it suffices to consider the case $P_3 = P_8 = P_{10}$, i.e., $-P = R \ast Q$ and $(-P) \ast (-P) = -P$.
Now assume for the contrary that $P_5 = P_7 = P_9$.
Then we have $-R = P \ast Q$ and $(-R) \ast (-R) = -R$, therefore $R \ast R = R$.
This implies that $R = (-P) \ast Q = ((-P) \ast (-P)) \ast Q$; by applying the case $P_2 = P_6 = P_7$ of Lemma \ref{lem:addition_on_E__2-6-7} to the right-hand side, we have
\[
((-P) \ast (-P)) \ast Q
= ((-Q) \ast (-P)) \ast P
= (-(P \ast Q)) \ast P
= R \ast P \enspace.
\]
Hence we have $R = R \ast P$, therefore $P = R \ast R = R$, contradicting Table \ref{tab:addition_on_elliptic_curve__equality_for_points}.
This implies that $P_5 = P_7 = P_9$ does not hold, therefore the claim follows from Lemma \ref{lem:addition_on_E__3-8-10_not_5-7-9}.
\end{proof}

By combining Lemmas \ref{lem:addition_on_E__3-8-9},  \ref{lem:addition_on_E__2-6-7}, and \ref{lem:addition_on_E__3-8-10}, the six possibilities listed above are exhausted and the claim $P_9 = P_{10}$ holds in any case.
This completes the proof of Theorem \ref{thm:rational_point_group}.

\end{document}